\documentclass[a4,11pt]{article}
\usepackage{amsmath,amsfonts,amsthm,latexsym,amssymb,mathrsfs}

\newtheorem{df}{Definition}[section]
\newtheorem{thm}[df]{Theorem}
\newtheorem{rema}[df] {Remark}
\newtheorem{lem}[df] {Lemma}
\def\sfstp{{\hskip-1em}{\bf.}{\hskip.5em}}
\begin{document}

\pagestyle{myheadings} \markboth{E. Boasso, V. Rako\v cevi\'c }
{\it E. Boasso, V. Rako\v cevi\' c}
\title{\bf  Characterizations of EP and normal Banach algebra elements
and Banach space operators }
 \author{ Enrico Boasso, Vladimir  Rako\v cevi\' c}
\date{   }
\maketitle

\setlength{\baselineskip}{14pt}
\begin{abstract} \noindent Several characterizations of EP and normal Moore-Penrose invertible Banach algebra elements 
will be considered. The Banach space operator case will be also studied.
The results of the present article will extend 
well known facts obtained in the frames of matrices and Hilbert space operators.\par
\vskip.2truecm
\noindent  \it Keywords: \rm EP element, normal element, Moore-Penrose inverse,
group inverse, Banach algebra\par
\vskip.2truecm
\noindent \it AMS classification: 15A09; 47A05
\end{abstract}

\section {\sfstp Introduction}
\
\indent In \cite{MD,MD0} EP and normal Moore-Penrose invertible elements were studied in the frame 
of rings with involution focusing on the pure algebraic structure of the objects under consideration.
In addition, these works extended several well known results obtained for matrices,  \cite{CT,BG}, 
and for Hilbert space operators,  \cite{DK,D}. The objective of the present article
is to characterize both EP and normal Moore-Penrose invertible elements in arbitrary
Banach algebras and EP and normal Moore-Penrose invertible Banach space operators. It is worth 
noticing that although contexts and arguments are different, above all because of the lack of
an involution,
results similar to the ones in the above mentioned papers will be presented.
Moreover,  the proofs of the results of this work give a 
new insight into the cases where an involution does exist (matrices, Hilbert space operators, $C^*$-algebras).
Furthermore,  the results considered also apply to EP and normal matrices
defined using an abitrary norm on a finite dimensional vector space, which extends and generalizes the results
known for EP and normal matrices defined using the conjugate transpose of a matrix.
\par

\indent From now on, $X$ will denote a Banach space and $L(X)$ the
Banach algebra of all bounded and linear maps defined on and with values in $X$.
In addition, if $T\in L(X)$, then $N(T)$ and $R(T)$ will stand for the null space and the
range of $T$ respectively. Note also that $I\in L(X)$ will denote the identity operator on $X$. \par

\indent Recall that the \it descent \rm and the \it ascent \rm of $T\in L(X)$ are
$d(T) =\hbox{ inf}\{ n\ge 0\colon  R(T^n)=R(T^{n+1})\}$ and
$ a(T)=\hbox{ inf}\{ n\ge 0\colon  N(T^n)=N(T^{n+1})\}$
respectively, where if some of the above sets is empty, its infimum is then defined as $\infty$, see for example \cite{T}.
In particular, note that if $a(T)$ and $d(T)$ are finite, then they coincide, see \cite[Theorem 3.6]{T}.\par
 \par
 
\indent On the other hand, $A$ will denote a unital Banach algebra
and $e\in A$ will stand for the identity of $A$. 
If $a\in A$, then $L_a \colon A\to A$
and $R_a\colon A\to A$ will denote the map defined by
left and right multiplication respectively:
$$
L_a(x)=ax, \hskip2truecm R_a(x)=xa,
$$   
where $x\in A$. Note that given $a$, $b\in A$,
$L_{ab}=L_aL_b$ and that $L_a=L_b$ implies that $a=b$.
Similarly, $R_{ab}=R_bR_a$ and if $R_a=R_b$, then $a=b$.
Moreover, the following notation will be used:
$$
N(L_a)= a^{-1}(0),\hskip.2truecm R(L_a)=aA,\hskip.2truecm
N(R_a)= a_{-1}(0),\hskip.2truecm R(R_a)=Aa.
$$
\indent Recall that an element $a\in A$ is called \it{regular}, \rm 
if it has a \it{generalized inverse}, \rm namely if there exists $b\in A$ such that
$$
a=aba.
$$

\indent Furthermore, a generalized inverse $b$ of a regular
element $a\in A$ will be called \it{normalized}, \rm if $b$ is regular
and $a$ is a generalized inverse of $b$, equivalently,
$$
a=aba, \hskip2truecm b=bab.
$$

\indent Note that if $b$ is a generalized inverse of $a$,
then $c=bab$ is a normalized generalized inverse
of $a$.\par

\indent Next follows the key notion in the definition of the 
Moore-Penrose inverse in context of Banach algebras.\par

\begin{df}\label{df1}Given a unital Banach algebra $A$, an element $a\in A$ will be said to be hermitian,
if $\parallel exp(ita)\parallel =1$ for all $ t\in\Bbb R$.
\end{df}

\indent As regard equivalent definitions and the main properties of hermitian 
Banach algebra elements see \cite{V} and \cite[pp. 55, 57, 67, 205]{BD}. Concerning  hermitian Banach space operators, see \cite[Chapter 4]{Dw}. 
Note that an element of a  $C^*$-algebra is 
hermitian if  and only if it is self-adjoint, see \cite[Proposition 20, Chapter I, Section 12]{BD}.\par

\indent In \cite{R1} V. Rako\v cevi\' c introduced the notion of Moore-Penrose invertible
Banach algebra elements. Next the definition of such objects will be recalled.\par

\begin{df}\label{df2} Let $A$ be a unital Banach algebra and consider $a\in A$. If there exists 
$x\in A$ such that $x$ is a normalized generalized inverse of $a$
satisfying that $xa$ and $ax$ are hermitian, then the element $x$ 
will be said to be the Moore-Penrose inverse of $a$, and it will be
denoted by $a^{\dag}$.
\end{df}

\indent In the conditions of Definition \ref{df2}, note that according to \cite[Lemma 2.1]{R1}, there is at most one Moore-Penrose inverse of $a\in A$.
Concerning the notion under consideration, 
see \cite{R1,R2,R3,Bo}. Note that according to  \cite[Proposition 20, Chapter I, Section 12]{BD},
in the frames of matrices with the conjugate transpose, Hilbert space operators and $C^*$-algebras, Definition \ref{df2}
coincides with the usual definition of the Moore-Penrose inverse.
For the original definition of the Moore-Penrose
inverse for matrices, see \cite{Pe}. In the following remark some of the most important properties
of the Moore-Penrose inverse in Banach algebras will be recalled.\par

 \begin{rema}\label{rema6}\rm (i) Let $A$ be a unital Banach algebra and consider $a\in A$.
If $a^{\dag}$ exists, then $a^{\dag}$ is Moore-Penrose invertible. In fact, according to 
Definition \ref{df2}, $(a^{\dag})^{\dag}=a$.\par
\indent (ii) In the conditions of (i), note that according to 
\cite[Theorem 5(ii)]{Bo}, given $a\in A$ a regular element, necessary and sufficient for $a$ to be Moore-Penrose invertible
is that $L_a\in L(A)$ has a Moore-Penrose inverse. Moreover, in this case $(L_a)^{\dag}=L_{a^{\dag}}$. \par
\indent (iii) Let $A=L(X)$, $X$ a Banach space, and $T\in L(X)$ a Moore-Penrose
invertible operator. Then  $TT^{\dag}\in L(X)$ and $T^{\dag}T\in L(X)$ are idempotents such that
$R(TT^{\dag})=R(T)$, $N(TT^{\dag})=N(T^{\dag})$, $X=R(T)\oplus N(T^{\dag})$, $R(T^{\dag}T)=R(T^{\dag})$,
$N(T^{\dag}T)=N(T)$ and $X=R(T^{\dag})\oplus N(T)$. In fact,
according to Definition \ref{df2}, $T=TT^{\dag}T$ and $T^{\dag}=T^{\dag}TT^{\dag}$.
Consequently, $T^{\dag}T=T^{\dag}TT^{\dag}T$, $TT^{\dag}=TT^{\dag}TT^{\dag}$,
$R(T)\subseteq R( TT^{\dag})$ and $N( TT^{\dag})\subseteq N(T^{\dag})$.
Since $R( TT^{\dag})\subseteq R(T)$ and $N(T^{\dag})\subseteq N( TT^{\dag})$,
$R(TT^{\dag})=R(T)$ and $N(TT^{\dag})=N(T^{\dag})$. In particular, $X=R(T)\oplus N(T^{\dag})$. The remaining identities
can be proved interchanging $T$ with $T^{\dag}$.
\end{rema}
\indent In the following definition the notion of EP Banach algebra element will be recalled, see
\cite{Bo}.\par

\begin{df}\label{d3}Given a unital Banach algebra $A$, an element $a\in A$ will be said to be EP, 
if there exists $a^{\dag}$, and  $aa^{\dag}= a^{\dag}a$. 
\end{df}

\indent As regard the properties of EP Hilbert space operators see \cite{Br,DKP},
of EP $C^*$-algebra elements see  \cite{HM,K,B,DKS}, and of EP Banach space
operators and EP Banach algebra elements see \cite{Bo}. 
In the following remark some of the main results concerning EP Banach algebra elements
will be considered.\par

 \begin{rema}\label{rema7}\rm (i)  Let $A$ be a unital Banach algebra and consider $a\in A$. Note that $a\in A$
is EP if and only if $a^{\dag}$ is EP. \par
\indent (ii) In the conditions of (i), according to \cite[Remark 12]{Bo}, necessary and
sufficient for $a\in A$ to be EP is the fact that $L_a\in L(A)$ is EP. \par
\indent (iii) Let $A=L(X)$, $X$ a Banach space, and consider $T\in L(X)$. Then, according to \cite[Theorem 16]{Bo},
$T$ is EP if and only if $R(T)=R(T^{\dag})$ or $N(T)=N(T^{\dag})$.
\end{rema}
\indent The group inverse is a notion closely related to the one of EP Banach algebra element. 
Since in the following section this notion will be intensively used,
it will be recalled.\par

\begin{df}\label{df4}Given a unital Banach algebra $A$ and $a\in A$, an element $b\in A$ will be said 
to be the group inverse of $a$, if the following set of equations is satisfied:\par
$$
a=aba, \hskip.8cm b=bab, \hskip.8cm ab=ba.
$$
\end{df}
\indent In the conditions of Definition \ref{df4}, note that according to \cite[Theorem 9]{HM1}, if the group inverse of $a\in A$ exists, then
it is unique. In this case, the group inverse of $a\in A$ will be denoted by $a^{\sharp}$.  In the following remark some of the most relevant properties of the group inverse will be
given.\par

\begin{rema}\label{rema5}\rm  (i)  Let $A$ be a unital Banach algebra and consider $a\in A$.
Suppose that $b\in A$ is a normalized generalized inverse of $a$. Then, necessary and sufficient for $b$ to be the group inverse of $a$ is that
$L_b\in L(A)$ (respectively $R_b\in L(A)$) is the group inverse of $L_a\in L(A)$ (respectively $R_a\in L(A)$). In fact, since $L_b$ is a normalized generalized inverse of $L_a$,
according to Definition \ref{df4} the statement under consideration is equivalent to saying that $a$ and $b$ commute 
if and only if $L_a$ and $L_b$ commute, which is clear. A similar argument proves the statement for the right multiplication operator on $L(A)$.
In addition, note that in this case, according to  \cite[Theorem 9]{HM1}, $(L_a)^{\sharp} =L_{a^\sharp}$ (respectively $(R_a)^{\sharp} =R_{a^\sharp}$).\par

\indent (ii) In the conditions of (i), note that if $a\in A$ is group invertible,
then necessary and sufficient for $a$ to be EP is that $a a^{\sharp}=a^{\sharp}a$
is a hermitian element. In fact, if $a\in A$ is EP, then according to 
\cite[Theorem 9]{HM1}, $a^{\sharp}$ exists, actually $a^{\sharp}=  a^{\dag}$.
In particular, $a a^{\sharp}=a^{\sharp}a=a^{\dag}a$ is hermitian.
On the other hand, if $a^{\sharp}$ exists and  $aa^{\sharp}=a^{\sharp}a$
is hermitian, then according to Definition \ref{df2}, $a^{\dag}$ exists.
What is more, according to  \cite[Lemma 2.1]{R1}, 
$a^{\dag}=a^{\sharp}$. Since
$a a^{\dag}=aa^{\sharp}= a^{\sharp}a=a^{\dag}a$,
$a$ is EP.\par

\indent (iii) Let $A=L(X)$, $X$ a Banach space, and consider $T\in L(X)$. Then, according to \cite[Lemma 1 and Theorem 4]{Ki}, 
the following statements are equivalent:
$$
(1) \hskip.3truecm T^{\sharp} \hbox{ exists}, \hskip.3truecm (2)\hskip.3truecm X=N(T)\oplus R(T),\hskip.3truecm (3)\hskip.3truecma(T)\le 1 \hbox{ and } d(T)\le 1.
$$ 
\indent (iv) In the conditions of (iii), note that if $T\in L(X)$ is group invertible, 
then an argument similar to the one in Remark \ref{rema6}(iii)
proves that $N(T)=N(T^{\sharp}T)=N(TT^{\sharp})=N(T^{\sharp})$ and $R(T)=R(TT^{\sharp})=R(T^{\sharp}T)=R(T^{\sharp})$.
 \end{rema}
\indent On the other hand, to prove the characterizations of Section 3, normal Banach algebra
elements need to be considered.\par

\indent Let $A$ be a unital Banach algebra and denote by $\mathcal{H}(A)$ the set of all
hermitian elements of $A$. Set $\mathcal{V}(A)=\mathcal{H}(A)+i\mathcal{H}(A)$.
Recall that according to \cite[Hilfssatz 2(c)]{V}, for each $a\in \mathcal{V}(A)$ there exist
necessary unique hermitian elements $u$, $v\in \mathcal{H}(A)$ such that $a=u+iv$. As a result,
 the operation $a^*=u-iv$ is well defined. Note that $^*\colon \mathcal{V}(A)\to \mathcal{V}(A)$
is not an involution, in particular $(ab)^*$ does not in general coincide with $b^*a^*$,
where $a$, $b\in \mathcal{V}(A)$.  However, if $A=\mathcal{V}(A)$ and for
every $h\in \mathcal{H}(A)$,  $h^2=u+iv$, where $u$, $v\in \mathcal{H}(A)$ and $uv=vu$,
then $A$ is a $C^*$-algebra
whose involution is the just considered operation,
see \cite{Be,G,V}. Next follows the definition of normal Banach algebra element. \par

\begin{df}\label{df10}Given a unital Banach algebra $A$, $a=u+iv\in  \mathcal{V}(A)$  will be said to be 
normal if $uv=vu$. 
\end{df}
\indent In the follwing remark several properties of normal Banach algebra elements will be recalled.\par
 
\begin{rema}\label{rema11}\rm (i) Let $A$ be a unital Banach algebra and consider $a\in  \mathcal{V}(A)$.
Note that necessary and sufficient for $a$ to be normal is the fact that $aa^*=a^*a$.\par
\indent (ii) In the conditions of (i), if $a\in A$, then according to the proof of \cite[Theorem 5]{Bo}, necessary and sufficient for $a\in A$
to belong to   $\mathcal{H}(A)$ is that $L_a\in  \mathcal{H}(L(A))$. Therefore, $a\in \mathcal{V}(A)$ is 
normal if and only if $L_a\in  \mathcal{V}(L(A))$
is normal. A similar statement can be proved if $R_a\in L(A)$ instead of $L_a\in L(A)$
is considered.  Note also that if $a\in \mathcal{V}(A)$, then $L_a$, $R_a\in  \mathcal{V}(L(A))$, and $(L_a)^*=L_{a^*}$ and $(R_a)^*=R_{a^*}$.\par
\indent (iii) When $A=L(X)$, $X$ a Banach space, if $T\in L(X)$ is a normal operator,
then according to \cite[Lemma 3]{DGS}, $N(T^*)=N(T)$. In addition, if $R(T)$ is closed, then
according to \cite[Corollary 4]{F}, $T^{\sharp}$ exists. Since $TT^*=T^*T$,
according to Remark  \ref{rema5}(iii), $R(T^*)\subseteq R(T)$.  \par 
\indent (iv) In the conditions of (i), when $a\in A$ is a normal element,
according to what has been recalled, it is not difficult to prove that
$a^{-1}(0)=(a^*)^{-1}(0)$, $a_{-1}(0)=(a^*)_{-1}(0)$,
and if $aA$ (respectively $Aa$) is closed, then
$a^*A\subseteq aA$ (respectively $Aa^*\subseteq Aa$).
\end{rema}

\indent Next a characterization of normal invertible elements will be presented.
The following Theorem presents a new proof of the main result in \cite{S}.\par

\begin{thm}\label{thm12}Let $A$ be a unital Banach algebra and consider $a\in  \mathcal{V}(A)$ a normal
element. Then, if $a$ is one side invertible, $a$ is invertible. 
\end{thm}
\begin{proof} Suppose that $a$ is right invertible but not invertible. Then, there exists $b\in A$ such that $ab=e$ but 
$0\in \sigma(a)$. Next consider $L_a\in L(A)$. Then, since according to \cite[Proposition 4(ii), Chapter I, Section 5]{BD} $\sigma(a)=\sigma(L_a)$, $0\in \sigma(L_a)$. 
Moreover,  according to Remark  \ref{rema11}(ii), $L_a\in L(A)$ is normal. Furthermore, since $ab=e$,
$aA=A$ is closed. Therefore, according to Remark \ref{rema11}(iii), $(L_a)^{\sharp}$ exists. However, 
according to Remark \ref{rema5}(iii), $A=a^{-1}(0)\oplus aA$. 
In particular, $a^{-1}(0)=0$ and then $L_a\in L(A)$ is an invertible map, which is impossible for $0\in \sigma (L_a)$.
Therefore, $a$ is invertible.\par

\indent A similar argument proves that it is impossible for $a$ to be left invertible but not
invertible.
\end{proof}

\section {\sfstp EP Banach algebra elements}
\
\indent In this section characterizations  of EP Banach algebra elements will be presented.
However, to this end some preliminary results must be considered.\par

\begin{lem}\label{lemA}Let $X$ be a Banach space and consider $T\in L(X)$
such that $T^{\dag}$ and $T^{\sharp}$ exist. Then the following statements
are equivalent.
\begin{align*} 
&(i)\hskip.2truecm T \hbox{ is EP },& &(ii)\hskip.2truecm R(T)\subseteq R(T^{\dag}),& &(iii)\hskip.2truecm R(T^{\dag})\subseteq R(T),&\\
&(iv)\hskip.2truecm  N(T)\subseteq N(T^{\dag}),& &(v)\hskip.2truecm N(T^{\dag})\subseteq N(T).& & &\\ 
\end{align*} 
\end{lem}
\begin{proof} According to Remark \ref{rema7}(iii), statement (i) implies statements (ii)-(v).\par
\indent Suppose that statement (ii) holds. Then, according to Remark  \ref{rema6}(iii)
and Remark \ref{rema5}(iii),
$$
X=R(T^{\dag})\oplus N(T)=R(T)\oplus N(T).
$$
Thus $N(T)$ is a common complement of $R(T^{\dag})$ and $R(T)$ and so 
$R(T)\subseteq R(T^{\dag})$ implies that $R(T^{\dag})= R(T)$ (see \cite[p.142]{DY} where
it is mentioned without proof that if $M$, $N$ are subspaces of $X$ and $N$ is a 
proper subspace of $M$, then $M$ and $N$ cannot have a common complement).
To prove the above mentioned implication, let $m\in R(T^{\dag})$. Next consider $n\in N(T)$
and $l\in R(T)$ such that $m=n+l$. Since  $R(T)\subseteq R(T^{\dag})$, $l-m=-n\in R(T^{\dag})\cap N(T)=0$.
In particular, $n=0$ and $m=l\in R(T)$. Therefore, $R(T^{\dag})=R(T)$, and according to Remark
\ref{rema7}(iii), $T$ is EP.\par 
\indent The equivalences among statement (i) and statements (iii)-(v) follow in a similar manner.
\end{proof}

\indent Next results of \cite[Theorem 4.1]{DK} will be extended from Hilbert space
operators to Banach algebra elements, 
see also \cite[Theorems 1, 3 and 4]{BG}, where matrices were considered.
Compare with \cite[Theorem 2.1]{MD0}, where EP elements  in rings with involution were studied.

\begin{thm}\label{thm6}Let $A$ be a unital Banach algebra and
consider $a\in A$ such that $a^{\dag}$ exists. Then $a$ is EP
if and only if  $a^{\sharp}$ exists and one of the following statements holds.\par
\begin{align*} 
(i)\hskip.2cm &aa^{\dag}a^{\sharp}=a^{\dag}a^{\sharp}a,& (ii)\hskip.2cm &aa^{\dag}a^{\sharp}=a^{\sharp}aa^{\dag},\\
(iii)\hskip.2cm &aa^{\sharp}a^{\dag}=a^{\dag}aa^{\sharp},& (iv)\hskip.2cm &aa^{\sharp}a^{\dag}=a^{\sharp}a^{\dag}a,\\
(v)\hskip.2cm &a^{\dag}aa^{\sharp}=a^{\sharp}a^{\dag}a, & (vi)\hskip.2cm &(a^{\dag})^2a^{\sharp}=a^{\dag}a^{\sharp}a^{\dag},\\
(vii)\hskip.2cm &aa^{\sharp}a^{\dag}=a^{\sharp},& (viii)\hskip.2cm & a^{\dag}a^{\sharp}=a^{\sharp}a^{\dag},\\
(ix)\hskip.2cm &a^{\dag}a^{\sharp}a^{\dag}=a^{\sharp}(a^{\dag})^2,& (x)\hskip.2cm & a^{\dag}(a^{\sharp})^2=a^{\sharp}a^{\dag}a^{\sharp},\\
(xi)\hskip.2cm &a^{\dag}(a^{\sharp})^2=(a^{\sharp})^2a^{\dag},& (xii)\hskip.2cm &(a^{\sharp})^2a^{\dag}=a^{\sharp} a^{\dag}a^{\sharp},\\
(xiii)\hskip.2cm &aa^{\sharp}=a^{\dag}a,& (xiv)\hskip.2cm &a^{\dag}a^{\dag}=a^{\dag}a^{\sharp},\\
(xv)\hskip.2cm &a^{\dag}a^{\dag}=a^{\sharp} a^{\dag},& (xvi)\hskip.2cm &(a^{\dag})^2=(a^{\sharp})^2,\\
(xvii)\hskip.2cm &a^{\dag}a^{\sharp}=(a^{\sharp})^2,& (xviii)\hskip.2cm &a(a^{\dag})^2=a^{\sharp},\\
(xix)\hskip.2cm &a^{\sharp}a^{\dag}a=a^{\dag},& (xx)\hskip.2cm &a^{\dag}aa^{\sharp}=a^{\dag}.&  \\
\end{align*}
\end{thm}

\begin{proof}

\indent It is clear that if $a$ is EP, then all the statement hold.\par

\indent On the other hand, to prove the converse implications, first of 
all consider $A=L(X)$, $X$ a Banach space, and let $T\in L(X)$ be
a bounded operator defined on $X$ such that $T^{\dag}$ and $T^{\sharp}$ exist.
The statements of the present Theorem  will be divided in several
cases.\par

\indent First case: $R(T)=R(T^{\dag})$.\par
\indent Suppose that statement (i) holds, that is
$TT^{\dag}T^{\sharp}=T^{\dag}T^{\sharp}T$. In particular, according to Remark \ref{rema5}(iv),
$$
R(TT^{\dag}T^{\sharp})=TT^{\dag}(R(T^{\sharp}))=TT^{\dag}(R(T))=R(TT^{\dag}T)=R(T).
$$
\indent On the other hand, according to Remark \ref{rema5}(iv) and Remark \ref{rema6}(iii),
 $$
R(T^{\dag}T^{\sharp}T)=T^{\dag}(R(T^{\sharp}T))=T^{\dag}(R(T))=R(T^{\dag}T)=R(T^{\dag}).
$$
\indent Therefore, $R(T)=R(T^{\dag})$, and according to Remark \ref{rema7}(iii), $T$ is an 
EP operator.\par
\indent Next suppose that statement (viii) is satisfied, that is $T^{\dag}T^{\sharp}=T^{\sharp}T^{\dag}$.
Consequently, according to Remark \ref{rema5}(iv) and Remark \ref{rema6}(iii),
$$
R(T^{\dag}T^{\sharp})=T^{\dag}(R(T^{\sharp}))=T^{\dag}(R(T))= R(T^{\dag}T)=R(T^{\dag}).
$$
As regard the other range space, according to Remark \ref{rema5}(iii)-(iv) and Remark \ref{rema6}(iii),
$$
R(T^{\sharp}T^{\dag})=T^{\sharp}(R(T^{\dag})\oplus N(T))= R(T^{\sharp})=R(T).
$$
As before, $T$ is EP. Moreover, the equivalence between statement (xiii) and the condition of
being EP can be proved in a similar way.\par

\indent Second case: $R(T^{\dag})\subseteq R(T)$.\par
\indent Suppose that statement (v) is true, that is $T^{\dag}TT^{\sharp}=T^{\sharp}T^{\dag}T$.
Then, according to Remark \ref{rema5}(iv) and Remark \ref{rema6}(iii),
$$
R(T^{\dag}TT^{\sharp})=T^{\dag}(R(TT^{\sharp}))=T^{\dag}(R(T))= R(T^{\dag}),
$$
while $R(T^{\sharp}T^{\dag}T)\subseteq R(T^{\sharp})=R(T)$ (Remark \ref{rema5}(iv)). Therefore, 
$R(T^{\dag})\subseteq R(T)$. However, according to Lemma \ref{lemA}(iii),  $T$ is EP.
The equivalences among statements (ix)-(xi) and the condition of being EP can be proved in a 
similar way.\par

\indent Third case: $R(T)\subseteq R(T^{\dag})$.\par

\indent Suppose that statement (xv) holds, that is $T^{\dag}T^{\dag}=T^{\sharp} T^{\dag}$.
Thus, according to Remark \ref{rema5}(iv) and Remark \ref{rema6}(iii),
$$
R(T^{\sharp} T^{\dag})=T^{\sharp}(R( T^{\dag}))=T^{\sharp}(R( T^{\dag})\oplus N(T))=R(T^{\sharp} )=R(T).
$$
Therefore, $R(T)\subseteq R((T^{\dag})^2)\subseteq  R(T^{\dag})$. Now well,
according to Lemma \ref{lemA}(ii), $T$ is EP.
The equivalences among statements (xvi)-(xvii) and the condition of being EP can be proved in a similar way.\par 
\indent Fourth case: $N(T)\subseteq N(T^{\dag})$.\par

\indent Suppose that statement (ii) holds, that is $TT^{\dag}T^{\sharp}=T^{\sharp}TT^{\dag}$.
Consider $x\in N(T)$. Then, $TT^{\dag}(x)\in N(T^{\sharp})=N(T)$ (Remark \ref{rema5}(iv)). Consequently,
$TT^{\dag}(x)\in N(T)\cap R(T)=0$ (Remark \ref{rema5}(iii)). In particular,  $x\in N(TT^{\dag})=N(T^{\dag})$ (Remark \ref{rema6}(iii)).
Therefore, $N(T)\subseteq N(T^{\dag})$. However, according to Lemma \ref{lemA}(iv),
 $T$ is EP.  The equivalences among statements (vi), (xii) and (xix)-(xx) and the condition of being EP 
can be proved in a similar way.\par

\indent The fifth and last case: $N(T^{\dag})\subseteq N(T)$.\par
\indent Suppose that statement (iii) holds, that is $TT^{\sharp}T^{\dag}=T^{\dag}TT^{\sharp}$.
Let $x$ belong to $N(T^{\dag})$. Then, $TT^{\sharp}(x)\in N(T^{\dag})\cap R(T)=0$ (Remark \ref{rema6}(iii)).
Thus, $x\in N(TT^{\sharp})=N(T)$ (Remark \ref{rema5}(iv)). Now well, since $N(T^{\dag})\subseteq N(T)$, according to
Lemma \ref{lemA}(v),
$T$ is EP.  The equivalences among statements (iv), (vii), (xiv) and (xviii) and the condition 
of being EP can be proved in a
similar way.\par

\indent Next consider an arbitrary Banach algebra $A$. According to Remark \ref{rema6}(ii) and Remark \ref{rema5}(i), 
it is possible to consider $L_a$, $L_{a^{\dag}}$,
 $L_{a^{\sharp}}\in L(A)$. What is more, $(L_a)^{\dag}=L_{a^{\dag}}$
and $(L_a)^{\sharp}=L_{a^{\sharp}}$. Then, if one of the 
statements holds, the same statement holds for $L_a$ and, according to what has been proved, $L_a$ is EP, 
which, according to Remark \ref{rema7}(ii), implies that $a$ is EP.
\end{proof}

\indent Next some results of \cite[Theorem 5.1]{D}, see also \cite[Theorem 2.3]{CT}, will be proved
in the context of Banach algebras.
Compare with \cite[Theorem 2.1]{MD0}, where EP elements  in rings with involution were studied.
However,  recall first that given a Banach algebra $A$ and $a\in A$, the element $a$
is said to be \it quasinilpotent\rm, if $\sigma (a)=\{ 0\}$, where $\sigma (a)$ denotes the
spectrum of $a$. In particular, if $a\in A$ is quasinilpotent, then $r(a)=0$,
where $r(a)$ stands for the \it spectral radius \rm of $a$, i. e., $r(a)$ = sup $\{\mid \lambda\mid\colon \lambda\in \sigma (a)\}$.\par

\begin{thm}\label{thm7} Let $A$ be a unital Banach algebra and
consider $a\in A$ such that $a^{\dag}$ exists. 
Then, necessary and suffcient for $a\in A$ to be EP is that one of the following statements holds.\par
$$
(i)\hskip.2truecm a^2a^{\dag} +a^{\dag}a^2=2a,\hskip.2truecm
(ii)\hskip.2truecm  (a^{\dag})^2a+a (a^{\dag})^2= 2a^{\dag},\hskip.2truecm 
(iii)\hskip.2truecm  a^{\sharp} \hbox{ exists and }a^{\dag}a^{\sharp}a+aa^{\sharp} a^{\dag}= 2a^{\dag}.
$$
\end{thm}
\begin{proof}
\indent Clearly, if $a\in A$ is EP, then all the statements hold.\par
\indent Note that statement (i) is equivalent to $a(aa^{\dag}-a^{\dag}a)=(aa^{\dag}-a^{\dag}a)a$.
Now well, according to \cite{Kl}, $aa^{\dag}-a^{\dag}a$ is quasinilpotent. However, according to \cite[Theorem 17, Chapter I, Section 10]{BD},
the spectral radius of $aa^{\dag}-a^{\dag}a$ coincides with $\parallel aa^{\dag}-a^{\dag}a\parallel$.
Therefore, $aa^{\dag}=a^{\dag}a$.\par
\indent To prove that statement (ii) implies that $a$ is EP,  apply (i)  interchanging $a$ with $a^{\dag}$.\par
\indent Suppose that statement (iii) holds. Multiplying by $a$ it is not difficult to obtain 
$$
aa^{\sharp} + aa^{\dag}= 2aa^{\dag}.
$$ 
Thus, $aa^{\sharp}= aa^{\dag}$. However, multiplying by $a^{\dag}$, $a^{\dag}aa^{\sharp}=a^{\dag}$.
Consequently, according to Theorem \ref{thm6}(xx), $a$ is EP.
\end{proof}
\indent In what follows, the ascent and the descent of a Banach space operators will be
used to characterize EP bounded and linear maps and EP Banach algebra elements.\par

\begin{thm}\label{thm8}Let $X$ be a Banach space, and
consider $T\in L(X)$ such that $T^{\dag}$ exists. Then $T$ is EP
if and only if one of the following statements holds.\par
\begin{align*} 
&(i)&&a (T)<\infty \hbox{ and } T(T^{\dag})^2= T^{\dag},&
&(ii)&&a (T^{\dag})<\infty \hbox{ and } T^{\dag}T^2= T,&\\
&(iii)&&d (T^{\dag})<\infty \hbox{ and } T^2T^{\dag}=T,&
&(iv)&&d (T) <\infty \hbox{ and } (T^{\dag})^2 T=T^{\dag},&\\ 
&(v)&&a (T) <\infty\hbox{ and }T^2T^{\dag} =T,&
&(vi)&&a (T^{\dag})<\infty \hbox{ and }  (T^{\dag})^2T=T^{\dag},&\\ 
&(vii)&&d (T)<\infty \hbox{ and}\hskip.2truecm T^{\dag}T^2=T,&
&(viii)&&d (T^{\dag})<\infty \hbox{ and}\hskip.0575truecmT(T^{\dag})^2= T^{\dag}.\\
\end{align*}
\indent Furthermore, necessary and sufficient for $T$ to be $EP$ is that $T^{\sharp}$ exists and one of the
following identities holds.
$$
(ix)\hskip.2truecm T=T^{\dag}T^2,\hskip.1truecm (x)\hskip.2truecm T=T^2T^{\dag},\hskip.1truecm (xi)\hskip.2truecm T^{\dag}=T(T^{\dag})^2,
\hskip.1truecm (xii)\hskip.2truecm T^{\dag}=(T^{\dag})^2 T. 
$$
\end{thm}

\begin{proof} It is clear that if $T$ is an EP operator, then all the statements hold.\par

\indent On the other hand, if statement (i) holds, then $T^2(T^{\dag})^2= TT^{\dag}$.
In particular, $R(T)=R(TT^{\dag})\subseteq R(T^2)$ (Remark \ref{rema6}(iii)). Therefore, $d (T)\leq 1$, and since
$a(T)$ and $d(T)$ are finite, according to \cite[Theorem 3.6]{T}, $a(T)\le 1$ and $d(T)\leq 1$.
Consequently, according to Remark \ref{rema5}(iii), $T^{\sharp}$ exists. However, since $R(T^{\dag})\subseteq R(T)$,
according to Lemma \ref{lemA}(iii), $T$ is EP.\par

\indent Suppose that statement (iii) holds and let $x\in N((T^{\dag})^2)$. Then
$0=T^2(T^{\dag})^2(x) =TT^{\dag}(x)$. Consequently, $x\in N(TT^{\dag})=
N(T^{\dag})$ (Remark \ref{rema6}(iii)). Thus $a(T^{\dag})\leq 1$. However, since $d(T^{\dag})$
is finite and $a(T^{\dag})\leq 1$, according to \cite[Theorem 3.6]{T}, 
$(T^{\dag})^{\sharp}$ exists. Now well, since $N(T^{\dag})\subseteq N(T)$,
according to Lemma \ref{lemA}(iv) and   Remark \ref{rema7}(i),
$T$ is EP.\par

\indent If statement (v) holds, then $d (T)\leq 1$. Consequently, using 
an argument similar to the one in the previous paragraphs, $T^{\sharp}$ exists. However, 
since $N(T^{\dag})\subseteq N(T)$, according to Lemma \ref{lemA}(v),
 $T$ is EP.\par

\indent If statement (vii) holds, then $a(T)\le 1$, and then as before, 
$T^{\sharp}$ exists. However, since $R(T)\subseteq R(T^{\dag})$,
according to Lemma \ref{lemA}(ii), $T$ is EP.\par

\indent To prove that statements (ii), (iv), (vi) and (viii) imply the condition of being EP, use what has been proved
and interchage $T$ with $T^{\dag}$.\par
\indent Concerning the second characterization, if $T$ is EP, it is clear that $T^{\sharp}$ exists and 
statements (ix)-(xii) hold. On the other hand, if $T$ has a group inverse and one of the statements (ix) - (xii)
holds, then, according to Remark \ref{rema5}(iii) and what has been proved,  $T$ is EP.
\end{proof}

\indent Next Theorem \ref{thm8} will be applied to prove new characterizations of
EP Banach algebra elements. \par

\begin{thm}\label{thm9}Let $A$ be a unital Banach algebra and
consider $a\in A$ such that $a^{\dag}$ exists. Then necessary
and sufficient for $a\in A$ to be EP
is that $a^{\sharp}$ exists and one of the following statements holds.
$$
(i)\hskip.2truecm a=a^{\dag}a^2,\hskip.1truecm (ii)\hskip.2truecm a=a^2a^{\dag},\hskip.1truecm (iii)\hskip.2truecm a^{\dag}=a(a^{\dag})^2,
\hskip.1truecm (iv)\hskip.2truecm a^{\dag}=(a^{\dag})^2 a. 
$$
\end{thm}
\begin{proof} According to  Remark \ref{rema6}(ii) and Remark \ref{rema5}(i), $(L_a)^{\dag}=L_{a^{\dag}}$ and $(L_a)^{\sharp}=L_{a^{\sharp}}$ respect-
ively.
Moreover, according to  Remark \ref{rema7}(ii), $a$ is EP if and only if $L_{a^{\dag}}$ is EP. 
To prove the Theorem it is then enough to apply Theorem \ref{thm8} to $L_a\in L(A)$.
\end{proof}

\section {\sfstp Normal Banach algebra elements}
\
\indent  In the present section Moore-Penrose invertible normal elements in 
arbitrary Banach algebras will be studied. In first place, a well known property will be considered.\par

\indent Recall that necessary and sufficient for a $m\times m$ complex matrix $M$ to be normal
is that $M$ is EP and $M^*M^{\dag}=M^{\dag}M^*$, where $M^*$ denotes the conjugate transpose of $M$,
see \cite[Lemma 1.1(d)]{CT} and the references mentioned there. In the contex of $C^*$-algebras,
a similar result holds for normal Moore-Penrose invertible elements of the algebra. For sake
of completeness, this fact will be proved.\par

\begin{thm}\label{thm19}Let $A$ be $C^*$-algebra and consider $a\in A$
such that $a^{\dag}$ exists. Then, necessary and sufficient for
$a$ to be normal is that $a$ is EP and $a^*a^{\dag}=a^{\dag}a^*$.
\end{thm}
\begin{proof} Let $a\in A$ be a normal Moore-Penrose invertible element. 
Then, according to \cite[Theorem 5]{HM1}, $a$ is EP and $a^*$ and $a^{\dag}$ commute.\par
\indent Now suppose that $aa^{\dag}=a^{\dag}a$ and $a^*a^{\dag}=a^{\dag}a^*$, 
and consider $v=aa^*-a^*a$. A straightforward calculation, using in particular that
$aa^{\dag}$ and $a^{\dag}a$ are self-adjoint, proves that $va^{\dag}=0$. 
Consequently, $a^{\dag}A\subseteq v^{-1}(0)$.
In addition, since according to \cite[Theorem 3.1(iv)]{K} $a^{-1}(0)=(a^*)^{-1}(0)$,
$a^{-1}(0)\subseteq v^{-1}(0)$. However, according to Remark \ref{rema6}(ii),
$L_a\in L(A)$ is Moore-Penrose invertible, what is more, according to
 Remark \ref{rema6}(ii)-(iii), $A= a^{\dag}A\oplus a^{-1}(0)$.
Therefore, $v^{-1}(0)=A$, equivalently, $a$ is normal.
\end{proof}

\indent To characterize normal Moore-Penrose invertible 
Banach algebra elements,
some prep-
aration is needed.\par

\begin{thm}\label{thm13}Let $X$ be a Banach space and consider $T\in L(X)$
such that $T^{\dag}$ exists and $T\in   \mathcal{V}(L(X))$. Then, the following statements
hold.\par
\noindent (i) $R(T^*)\subseteq R(T)$ if and only if $T=TTT^{\dag}$.\par
\noindent (ii) $N(T)\subseteq N(T^*)$ if and only if $T=T^{\dag}TT$.\par
\indent In addition, necessary and sufficient for $T$ to be EP is that the conditions of statements (i) and (ii) hold.
\end{thm}
\begin{proof} Since $R(T)=R(TT^{\dag})=N(I-TT^{\dag})$,
$R(T^*)\subseteq R(T)$ is equivalent to $T^*=TT^{\dag}T^*$. \par
\indent Next consider $U$, $V\in \mathcal{H}(L(X))$ such that 
$T=U+iV$ and $T^*=U-iV$. Recall also that $T=TT^{\dag}T$.
Then, adding and substracting $T$ and $T^*$, $U=TT^{\dag}U$ and $V=TT^{\dag}V$.
In addition, according to \cite[Theorem 2.13]{Be1}, $UTT^{\dag}=TT^{\dag}U$
and $VTT^{\dag}=TT^{\dag}V$.   Therefore $T=U+iV=UTT^{\dag} +iVTT^{\dag}
=TTT^{\dag}$.\par
\indent On the other hand, if $T=TTT^{\dag}$, then 
$$
U+iV=(U+iV)TT^{\dag} =TT^{\dag}(U+iV).
$$
In particular $(UTT^{\dag}-TT^{\dag}U) + i(VTT^{\dag}-TT^{\dag}V)=0$.
However, since  $V$, $TT^{\dag}\in \mathcal{H}(L(X))$, according to 
\cite[Hilfssatz 2(b)]{V}, $i(VTT^{\dag}-TT^{\dag}V)\in \mathcal{H}(L(X))$.
In addition, according to \cite[Hilfssatz 2(a)]{V},
 $(UTT^{\dag}-TT^{\dag}U)=- i(VTT^{\dag}-TT^{\dag}V)\in \mathcal{H}(L(X))$.
Moreover, multiplying by $-i$ the identity in the third line of the present
paragraph, $(VTT^{\dag}-TT^{\dag}V) +i(TT^{\dag}U-UTT^{\dag})=0$.
Then, an argument similar to the previous one proves that  
$(VTT^{\dag}-TT^{\dag}V)\in \mathcal{H}(L(X))$. However, according to  \cite[Hilfssatz 2(c)]{V},
$UTT^{\dag}=TT^{\dag}U$ and $VTT^{\dag}=TT^{\dag}V$. Consequently, since $TT^{\dag}$ is an idempotent, $R(T)=R(TT^{\dag})$
and $N(T^{\dag})=N(TT^{\dag})$ are closed invariant subspaces both for $U$ and 
$V$.  \par

\indent Consider $U'=U\mid_{N(T^{\dag})}^{N(T^{\dag})}\in L(N(T^{\dag}))$ 
and $V'=V\mid_{N(T^{\dag})}^{N(T^{\dag})}\in L(N(T^{\dag}))$.
According to \cite[Proposition 4.12]{Dw}, $U'$, $V'\in  \mathcal{H}(L(N(T^{\dag})))$. If
$T'$ is the restrictions of $T$ to $N(T^{\dag})$,
then it is clear that  $T'=U'+iV'$.
However, since $T=TTT^{\dag}$, $N(T^{\dag})\subseteq N(T)$, which, according to 
 \cite[Hilfssatz 2(c)]{V}, implies that $U'=V'=0$. In particular,
$T^*(N(T^{\dag}))=0$. In addition, according to what has been proved
at the end of the previous paragraph, it is clear that $T^*(R(T))\subseteq R(T)$.
Therefore, accoding to Remark \ref{rema6}(iii), $R(T^*)\subseteq R(T)$.
\par

\indent Next suppose that  $N(T)\subseteq N(T^*)$. Since $N(T)=N(T^{\dag}T)=
R(I-T^{\dag}T)$,  $N(T)\subseteq N(T^*)$ is equivalent to
$T^*=T^*T^{\dag}T$. Now well, as in the proof of  statement (i), if $T=U+iV$
and $T^*=U-iV$, with $U$, $V\in  \mathcal{H}(L(X))$, adding and substracting $T$ and $T^*$ and using that $T=TT^{\dag}T$, 
it is then clear that $U=UT^{\dag}T$ and $V=VT^{\dag}T$. In particular, $UT^{\dag}T$,
$VT^{\dag}T\in  \mathcal{H}(L(X))$. However,
according again to \cite[Theorem 2.13]{Be1}, $UT^{\dag}T=T^{\dag}TU$
and $VT^{\dag}T=T^{\dag}TV$. Consequently 
$T=(U+iV)T^{\dag}T=T^{\dag}TT$.\par

\indent On the other hand, if $T=T^{\dag}TT$, applying an argument similar
to the one used in the proof of statement (i) but considering 
$T^{\dag}T$ instead of $TT^{\dag}$, it is then not difficult to prove that
$UT^{\dag}T=T^{\dag}TU$ and $VT^{\dag}T=T^{\dag}TV$. As a result,
since $T^{\dag}T$ is an idempotent such that $N(T^{\dag}T)=N(T)$, $N(T)$ is a closed invariant
subspaces both for $U$ and $V$.\par

\indent Consider $\tilde{U}=U\mid_{N(T)}^{N(T)}$, $\tilde{V}=V\mid_{N(T)}^{N(T)}\in L(N(T))$.
According again to \cite[Proposition 4.12]{Dw}, $\tilde{U}$, $\tilde{V}\in \mathcal{H}(L(N(T)))$.
However, since $T(N(T))=0$, according to  \cite[Hilfssatz 2(c)]{V}, $\tilde{U}=\tilde{V}=0$.
Consequently, $T^*(N(T))=0$, equivalently, $N(T)\subseteq N(T^*)$.\par

\indent The last statement is a consequence of \cite[Theorem 18(xii)]{Bo}.
\end{proof}

\indent Next normal Moore-Penrose invertible Banach algebra elements will be characterized.
Compare with  \cite[Proposition 27]{Ha} and \cite[Lemma 1.2]{MD} where normal Moore-Penrose invertible elements of rings with involution
were considered.\par

\begin{thm}\label{thm14} Let $A$ be a unital Banach algebra and consider $a\in \mathcal{V}(A)$
such that $a^{\dag}$ exists. Then, necessary and sufficient for $a$ to be normal is the fact that
$a$ is EP and $a^{\dag}a^*=a^*a^{\dag}$.
\end{thm}
\begin{proof} First of all suppose that $a$ is normal. According to Remark \ref{rema6}(ii),
$L_a\in L(A)$ has a Moore-Penrose inverse, what is more $(L_a)^{\dag}=L_{a^\dag}$,  and according to
Remark  \ref{rema11}(ii), $L_a\in \mathcal{V}(L(A))$ and $L_a$ is normal. 
However, according to  Remark  \ref{rema11}(iii) and Theorem  \ref{thm13},
$L_a$ is an EP operator, which, according to Remark \ref{rema7}(ii), is equivalent to the fact
that $a$ is EP. Furthermore, according to \cite[Theorem]{AS}, $a^{\dag}a^*=a^*a^{\dag}$.\par

\indent On the other hand, suppose that $a$ is EP and  $a^{\dag}a^*=a^*a^{\dag}$.
Since, according to Remark \ref{rema11}(ii) and Remark  \ref{rema6}(ii),  $(L_a)^*=L_{a^*}$ and $(L_a)^{\dag}
=L_{a^\dag}$ respectively, according to Remark  \ref{rema7}(ii), it is clear that
$L_a\in L(A)$ is EP and $(L_a)^{\dag}(L_a)^*=(L_a)^*(L_a)^{\dag}$. It will be proved that
$L_a(L_a)^*=(L_a)^*L_a$, which, according to Remark \ref{rema11}(i)-(ii), is equivalent to the fact that
$a$ is normal.\par
\indent Since $L_a$ is EP, according to  \cite[Theorem18(xii)]{Bo} and Theorem \ref{thm13}(i), 
$R((L_a)^*)\subseteq R(L_a)$. In particular, $R(L_a(L_a)^*-(L_a)^*L_a)\subseteq R(L_a)$.
According again to \cite[Theorem18(xii)]{Bo} and Theorem \ref{thm13}(ii), 
$(L_a(L_a)^*-(L_a)^*L_a)(N(L_a))=0$. Moreover, since
$(L_a)^{\dag}(L_a(L_a)^*-(L_a)^*L_a)(L_a)^{\dag}=0$,
$(L_a(L_a)^*-(L_a)^*L_a)(R(L_a)^{\dag})\subseteq N((L_a)^{\dag})$.
Thus, according to Remark \ref{rema6}(iii), $R(L_a(L_a)^*-(L_a)^*L_a)
\subseteq N((L_a)^{\dag})$. Therefore, since according to Remark \ref{rema6}(iii) $R(L_a)\cap N((L_a)^{\dag})=0$,
$L_a(L_a)^*-(L_a)^*L_a=0$.
\end{proof}
\indent In the following theorem results of \cite[Theorem 3.2]{DK} will be extended to normal Moore-Penrose invertible 
Banach algebra elements, see also  \cite[Theorems 2 and 6]{BG} and \cite[Theorem 2.2]{MD}.   \par

\begin{thm}\label{thm15}Let $A$ be a unital Banach algebra and consider $a\in \mathcal{V}(A)$ such that 
$a^{\dag}$ exists. Then, necessary and sufficient for $a$ to be  normal is the fact that $a^{\sharp}$ exists and 
one of the following conditions holds.
\begin{align*} 
&(i)& &aa^*a^{\sharp}=a^*a^{\sharp}a \hbox{ and } a=a^{\dag}aa,& &(ii)& 
&aa^*a^{\sharp}= a^{\sharp}aa^* \hbox{ and } a=aaa^{\dag},&\\
&(iii)& &aa^{\sharp}a^*=a^{\sharp}a^* a  \hbox{ and } a=aaa^{\dag},&
 &(iv)& &a^*aa^{\sharp}=a^{\sharp}a^*a \hbox{ and } a=a^{\dag}aa,&\\
&(v)& &a^{\dag}a^*a^{\sharp}=a^{\dag}a^{\sharp}a^* \hbox{ and } a=aaa^{\dag },&
&(vi)&&a^*a^{\sharp}a^{\dag }=a^{\sharp }a^*a^{\dag }\hbox{ and } a=a^{\dag }aa,&\\
&(vii)& &a^*=aa^*a^{\sharp },& &(viii)& &a^*=a^{\sharp }a^*a,&\\
&(ix)& &aa^*a=a^*aa \hbox{ and } a=a^{\dag}aa,& &(x)& 
&aaa^*= aa^*a \hbox{ and } a=aaa^{\dag}.&\\
\end{align*} 
\end{thm}
\begin{proof}According to Remark \ref{rema11}(ii), if $a\in \mathcal{V}(A)$ is normal, then $L_a\in L(A)$ is normal. 
In addition, since $a^{\dag}$ exists, according to Remark \ref{rema6}(ii), $(L_a)^{\dag}$ exists,
in particular $R(L_a)$ is closed. Consequently, according to Remark  \ref{rema11}(iii), $(L_a)^{\sharp}$
exists, which, according to Remark \ref{rema5}(i), implies that $a^{\sharp}$ exists. Moreover, since $a$ is a normal Moore-Penrose invertible 
element, according to Remark \ref{rema5}(ii), Theorem \ref{thm14} and \cite[Theorem 18(xii)]{Bo}, all the statements hold.\par

\indent To prove the converse implications, first of all the case $A=L(X)$, $X$ a Banach space, and $T\in L(X)$ 
such that $T^{\dag }$ and $T^{\sharp }$ exist will be considered.\par

\indent Suppose that  $TT^*T^{\sharp}=T^*T^{\sharp}T$ and  $T=T^{\dag}TT$. Then, according to Remark \ref{rema5}(iv), 
$(TT^*-T^*T)(R(T))=0$, and according to Theorem \ref{thm13}(ii), $(TT^*-T^*T)(N(T))=0$. However, according to Remark \ref{rema5}(iii),
$T$ is normal.\par

\indent If statement (ii) holds, then $R(T^*T^{\sharp}-T^{\sharp}T^*)\subseteq N(T)$. In addition,
according to Theorem  \ref{thm13}(i), $R(T^*)\subseteq R(T)$. Then, 
$R(T^*T^{\sharp}-T^{\sharp}T^*)\subseteq R(T)$. 
Consequently, according to Remark \ref{rema5}(iii), $T^*T^{\sharp}=T^{\sharp}T^*$. However, according to
Theorem  \ref{thm8}(x), $T$ is EP. Then, according to Remark \ref{rema5}(ii) and Theorem  \ref{thm14}, $T$ is normal.\par

\indent If $TT^{\sharp}T^*=T^{\sharp}T^* T$ and $T=TTT^{\dag}$, then,
according to Remark \ref{rema5}(iv),
$R(TT^*-T^*T)\subseteq N(T^{\sharp})=N(T)$. In addition, according to 
Theorem  \ref{thm13}(i), $R(T^*)\subseteq R(T)$. Then $R(TT^*-T^*T)\subseteq R(T)$.
However, according to Remark \ref{rema5}(iii), $T$ is normal.\par

\indent Next suppose that statement (iv) holds. According to  
Remark \ref{rema5}(iii)-(iv), $R(T^*T)\subseteq R(T)$.
In addition, note that since
$TT^{\sharp}(T^*T)= (TT^*)TT^{\sharp}$, 
$$
TT^{\sharp}(T^*T)TT^{\sharp}= TT^{\sharp}(TT^*)TT^{\sharp}.
$$
However, according to Theorem  \ref{thm13}(ii), $N(T)\subseteq N(T^*)$. Therefore,
according to Remark \ref{rema5}(iii), $TT^*=T^*T$.\par

\indent Concerning statement (v), it can be proved that $T$ is normal as in the case of statement (ii) using that $N(T^{\dag})\cap R(T)=0$
(Remark \ref{rema6}(iii)).\par

\indent Suppose that statement (vi) holds. Then $R(T^{\dag })\subseteq N(T^*T^{\sharp} -T^{\sharp }T^*)$.
Moreover, according to Remark \ref{rema5}(iv) and Theorem  \ref{thm13}(ii), $N(T)\subseteq N(T^*T^{\sharp} -T^{\sharp }T^*)$. However,
according to Remark \ref{rema6}(iii), $T^*T^{\sharp} =T^{\sharp }T^*$. Now proceed as in statement (ii).\par

\indent If statement (vii) holds, then according to Remark \ref{rema5}(iv), $N(T)\subseteq N(T^*)$.
Clearly, $R(T^*)\subseteq R(T)$. Therefore, according to Theorem  \ref{thm13}, $T$ is EP. 
In particular, according to Remark \ref{rema5}(ii), $T^{\sharp }=T^{\dag}$.
To prove that $T$ is normal, according to Theorem  \ref{thm14}, it is enough to show that $T^*T^{\dag}=T^{\dag}T^*$. \par
\indent Note that  according to Remark  \ref{rema7}(iii),
$R(T^*T^{\dag}-T^{\dag}T^*)\subseteq R(T)$. In addition, according to the proof of Theorem  \ref{thm13}(i),
since $R(T^*)\subseteq R(T)$, $T^*=TT^{\dag}T^*$. Now well, 
$$
T(T^*T^{\dag}-T^{\dag}T^*)= TT^*T^{\dag} -TT^{\dag}T^*=T^* -TT^{\dag}T^*=0.
$$
Therefore, $R(T^*T^{\dag}-T^{\dag}T^*)\subseteq N(T)$. According to Remark  \ref{rema5}(iii),
$T^*T^{\dag}=T^{\dag}T^*$.\par

\indent Suppose that statement (viii) holds. Then, it  can be proved that $T$ is normal as in the case of statement (vii) using 
$N(T^{\dag})\cap R(T)=0$ (Remark  \ref{rema6}(iii)).\par
\indent If statement (ix) holds, then it can be proved that $T$ is normal as in the case of statement (i) using $T$ instead of 
$T^{\sharp}$.\par

\indent Concerning statement (x), it can be proved that $T$ is normal as in the case of statement (iii) using $T$ instead of $T^{\sharp}$.\par

\indent Next consider an arbitrary Banach algebra $A$ and $a\in A$ satisfying the hypothesis
of the Theorem. According to Remark  \ref{rema6}(ii) and Remark  \ref{rema5}(i),
$L_a\in L(A)$ satisfies the hypotesis of the Theorem. What is more, according to Remark  \ref{rema5}(i), 
 Remark \ref{rema6}(ii) and Remark  \ref{rema11}(ii), 
 if $a$ satisfies one of the
statements of the Theorem, then $L_a$ satisfies the same statement. Consequently, according
to what has been proved, $L_a\in L(A)$ is normal. Then, according to Remark  \ref{rema11}(ii), $a$ is normal. 
\end{proof}

\begin{rema}\label{rema16}\rm Given $A$ a $C^*$-algebra, it is well known that
if $a$ and $x\in A$, then $a^*ax=0$ implies that $ax=0$. Similarly, $xaa^*=0$
implies $xa=0$. More generally, if $\mathcal{R}$ is a ring with involution,
an element $a\in \mathcal{R}$ that satisfies this property is said to be
\it $*$-cancellable\rm, see for example  \cite[Definition 5.2]{KP} and \cite[Definition 1.1]{MD}.
In particular, if $a\in  \mathcal{R}$ is Moore-Penrose invertible, see \cite{KP, MD},
then according to \cite[Theorem 5.3]{KP}, $a$ is $*$-cancellable. However, in the 
context of the present article, not only $^*\colon \mathcal{V}(A)\to
\mathcal{V}(A)$ is not in general an involution, but also it is not clear if the cancellation 
property holds for Moore-Penrose invertible elements of $\mathcal{V}(A)$, see the proof of
\cite[Theorem 5.3]{KP}. The conjecture is that an $a\in \mathcal{V}(A)$
such that $a^{\dag}$  exists is not in general $*$-cancellable. \par

\indent Anyway, concerning equivalent statements characterizing the condition of
being normal, in the frame of a $C^*$-algebras or more generally in a ring with involution
$ \mathcal{R}$, the cancellation property and the identity $(ab)^*=b^*a^*$,
$a$ and $b\in  \mathcal{R}$, are two important properties, see the proof of
\cite[Theorem 2.2]{MD}. Since in the conditions of Theorem  \ref{thm15} these two 
properties fail to hold, in most statements an additional condition needs to be 
considered. Note that in these cases, according to Theorem  \ref{thm9}, 
an element $a$ that satisfies the conditions of Theorem  \ref{thm15} is EP.
Compare with Theorem  \ref{thm14}.
\end{rema}
\vskip.5truecm
\noindent {\bf Acknowledgments} The authors wish to express their indebtedness
to the referees, for their observations and suggestions
considerably improved the final version of the present article.\par

\vskip.5truecm

\noindent Enrico Boasso\par
\noindent E-mail address: enrico\_odisseo@yahoo.it

\vskip.3truecm
\noindent Vladimir Rako\v cevi\'c\par
\noindent University of Ni\v s - Faculty of Sciences and Mathematics \par
\noindent Visegradska 33 - 18000, Ni\v s - Serbia\par
\noindent E-mail address:

\end{document}